\documentclass[10pt,reqno]{amsart}
\usepackage{amsmath, latexsym, amsfonts, amssymb,amsthm, amscd,epsfig,enumerate}
\usepackage{subfigure,color, float}

\advance\hoffset -.75cm

\usepackage[usesnames]{xcolor}

\definecolor{refkey}{gray}{.75}
\definecolor{labelkey}{gray}{.75}

\newcommand{\Z}{\mathbb Z}
\newcommand{\N}{\mathbb N}
\newcommand{\Prob}{\mathbb P}

\newcommand{\cH}{\mathcal{H}}

\newcommand{\pr}{\mathbb P}

\newtheorem{teo}{Theorem}[section]
\newtheorem{lem}[teo]{Lemma}

\newtheorem{rem}[teo]{Remark}
\newtheorem{pro}[teo]{Proposition}
\newtheorem{defn}[teo]{Definition}
\newtheorem{exmp}[teo]{Example}

\title
{Branching random walks with uncountably many \\ extinction probability vectors}

\author[D.~Bertacchi]{Daniela Bertacchi}
\address{D.~Bertacchi, Dipartimento di Matematica e Applicazioni,
Universit\`a di Milano--Bicocca,
via Cozzi 53, 20125 Milano, Italy.}
\email{daniela.bertacchi\@@unimib.it}

\author[F.~Zucca]{Fabio Zucca}
\address{F.~Zucca, Dipartimento di Matematica,
Politecnico di Milano,
Piazza Leonardo da Vinci 32, 20133 Milano, Italy.}
\email{fabio.zucca\@@polimi.it}

\date{}

\begin{document}

\begin{abstract}
Given a branching random walk on a set $X$, we study its extinction probability vectors $\mathbf q(\cdot,A)$. Their components
are the probability that the process goes extinct in a fixed $A\subseteq X$, when starting from a vertex $x\in X$.
The set of extinction probability vectors (obtained letting $A$ vary among all subsets of $X$) is a
subset of the set of the fixed points of the generating function of the branching random walk. 
In particular here we are interested in the cardinality of the set of extinction probability vectors.
We prove results which allow to understand whether the probability of extinction in a set $A$ is
different from the one of extinction in another set $B$.
In many cases there are only two possible extinction probability vectors and 
so far, in more complicated examples, only a finite number of distinct extinction probability vectors 
had been explicitly found.
Whether a branching random walk could have an infinite number of distinct extinction probability vectors was not
known. We apply our results to construct examples of branching random walks with uncountably many distinct
extinction probability vectors.
\end{abstract}


\maketitle


\section{Introduction}
\label{sec:intro}

The branching random walk (or briefly $\textit{BRW}$) on an at most countable set $X$
is a process which describes the evolution of a population breeding and dying on $X$.
When $X$ is a singleton, the BRW reduces to the \textit{branching process} (the BRW is 
also known as \textit{multi-type branching process}). In the long run, for any $A\subseteq X$,
a BRW starting with one individual at $x\in X$ can go extinct in $A$ (no individuals alive in 
$A$ from a certain time on) or survive in $A$ (infinitely many visits to $A$).
If the probability of extinction in $A$ is equal to 1, we say that there is extinction in $A$,
otherwise that there is survival in $A$.

Letting $x$ vary in $X$ we get an extinction probability vector, and letting also $A$ vary we have
the family of extinction probability vectors.
These vectors are of particular interest and can be seen as fixed points of a suitable generating function 
 associated to the process (see Section~\ref{sec:basic}).
 
It is well known that the generating function of the branching process has at most 
two fixed points: the extinction probability and 1. When the space $X$ is not a singleton, there are, in principle, more extinction
probabilities (see Section~\ref{sec:basic}), thus more fixed points. It remains
true that the vector $\mathbf 1$ is always a fixed point and the \textit{global extinction}
probability vector (that is, the probability of extinction in the whole space $X$) is always the minimal fixed point.

In order to construct a BRW with a large number of fixed points, a trivial way is to
use \textit{reducible} BRWs (see Section~\ref{sec:basic} for the definition).
Roughly speaking, in the reducible case, $X$ is divided into classes and the progeny of particles living in some 
classes cannot colonize other ones, and it is not difficult to have different extinction probability vectors in each class.
The interesting case is when the BRW is \textit{irreducible}, i.e.~the progeny of a particle has always a positive probability of reaching every 
point of $X$. Therefore our study addresses the question whether an irreducible BRW can have infinitely many distinct extinction probability
vectors.

It turns out that for any irreducible BRW on a finite $X$ the situation is the same as in a branching process:
there are at most two fixed points. In particular this means that the value of the probability of extinction in a (nonempty) set $A$
does not depend on $A$. The interests on
fixed points slowed down when in \cite{cf:Spataru89} the author claimed that the generating function
of every irreducible BRW on a finite or countable $X$ has at most two fixed points.
In particular, if this were true, in any BRW the probability of extinction in $A \subseteq X$ would be either 1
or equal to the probability of global extinction (this last case is called strong local survival at $A$, see 
\cite{cf:BZ14-SLS, cf:MachadoMenshikovPopov}).

In \cite{cf:BZ14-SLS}
the authors found a gap in the proof of \cite[Theorem 3]{cf:Spataru89} and provided an example of
an irreducible BRW 
with more than two fixed points.
Since the,n some authors have been trying to describe the properties of the space of fixed points and the
subspace of extinction probabilities (see Section~\ref{sec:open} for further details on the state of the art). 

The first example of a BRW with (at least) two nontrivial extinction probabilities is in \cite[Example 4.2]{cf:BZ14-SLS}
while the first example of a BRW with (at least) three nontrivial extinction probabilities can be found  in \cite{cf:BraHautHessemberg2}.
Both BRWs are inhomogeneous, that is, the reproduction laws differ from site to site. 
It is natural to wonder whether a BRW can have infinitely many extinction probability vectors and in particular if this is
possible without ``inhomogeneity''. Perhaps the strongest ``homogeneity'' one can think of is transience, that is when
one can map every site to any other site through an automorphism (see Section~\ref{sec:basic}).
In this paper we provide an example of a transient BRW on a regular tree with an uncountable number 
of distinct extinction probabilities (see Section~\ref{exmp:tree}). 
The key to the proof is finding an uncountable family $\mathcal F$ of unions of subtrees such that the probability
that the process starting from the root goes extinct in $A$,
is different for all $A\in \mathcal F$.
Two members of $\mathcal F$ have different ``size of their boundaries'', so at first one may think that the 
probability of extinction, starting from the root, only depend on this size, but we prove that this is not the case
(again see Section~\ref{exmp:tree}).

In the example in Section~\ref{exmp:tree} there is extinction in all finite sets;
this is not necessary, indeed Example~\ref{exmp:modifiedBRW} is an inhomogeneous BRW on the tree with uncountably many extinction
probability vectors and survival with positive probability in each finite set.
The tree is a graph with a particularly ``large'' boundary, but even this property is not necessary:
in  Section~\ref{subsec:comb} we generalize our result to a wider class of BRWs with uncountably 
many extinction probabilities, supported on the comb,
which is a subset of $\Z^2$.

In Section~\ref{sec:basic} we define the process in discrete-time and in continuous-time,
survival and extinction and then define the generating function of a BRW (Section~\ref{subsec:genfun}), its fixed
points and the extinction probability vectors (Section~\ref{sec:extprobvect}).
We state and prove two results which tell when two extinction probabilities are equal or not. In particular
Theorem~\ref{th:strongconditioned} compares two BRWs whose
reproduction laws are the same outside a set $A \subset X$, while Lemma~\ref{lem:smallersurvival}
deals with the case when the process survives in a set, without ever visiting another set.
%
In Section~\ref{exmp:tree} we prove in detail that on the tree the generating function
of a BRW can have uncountably many extinction probabilities. This is not just a property of the tree; indeed, by projecting
 BRWs (see Definition~\ref{def:projection}) we prove in Section~\ref{subsec:comb}, that the same holds for a larger class (up
 to projections) of BRWs. Finally, Section~\ref{sec:open} is devoted to a brief description of the state of the art on extinction
 probabilities and fixed points and contains some open questions.

\section{Extinction probabilities: definitions and properties}
\label{sec:basic}

Given an at most countable set $X$, a discrete-time BRW is a process $\{\eta_n\}_{n \in \N}$, 
where $\eta_n(x)$ is the number of particles alive
at $x \in X$ at time $n$. 
The dynamics is described as follows:
consider the (countable) measurable space $(S_X,2^{S_X})$
where $S_X:=\{f:X \to \N\colon \sum_yf(y)<\infty\}$ and let 
$\mu=\{\mu_x\}_{x \in X}$ be a family of probability measures
on $(S_X,2^{S_X})$.
A particle of generation $n$ at site $x\in X$ lives one unit of time;
after that, a function $f \in S_X$ is chosen at random according to the law $\mu_x$.
This function describes the number of children and their positions, that is,
the original particle is replaced by $f(y)$ particles at
$y$, for all $y \in X$. The choice of $f$ is independent for all breeding particles.
The BRW is denoted by $(X,\mu)$ and it is a Markov chain with absorbing state $\mathbf 0$,
the configuration with no particles at all sites.

The total number of children associated to $f$ is represented by the
function $\cH:S_X \rightarrow \N$ defined by $\cH(f):=\sum_{y \in X} f(y)$;
the associated law $\rho_x(\cdot):=\mu_x(\cH^{-1}(\cdot))$ is the law of the random number of children
of a particle living at $x$.
We denote by
 $m_{xy}:=\sum_{f\in S_X} f(y)\mu_x(f)$ 
 the expected number of children that a particle living
 at $x$ sends to $y$. It is easy to show that $\sum_{y \in X} m_{xy}=\bar \rho_x$
 where $\bar \rho_x$ is the expected value of the law $\rho_x$.
%


In particular, if $\rho_x$ does not depend on $x \in X$, we say that the BRW can be
\textit{projected on a branching process} (see~\cite{cf:BCZ, cf:BZ4} for details).
More generally, 
some BRWs can be projected onto BRWs defined on finite sets as explained in \cite[Section 2.3]{cf:BZ2017}
(see also Section~\ref{def:projection} for the case of continuous-time BRWs).
In the case of the projection on a branching process, the finite set is a singleton.
Other examples are the so called \textit{quasi-transitive BRWs} 
(see \cite[Section 2.4]{cf:BZ14-SLS} for the formal definition) where the action of the group of the automorphisms of 
the BRW (namely, bijective maps preserving the reproduction laws) has a finite number $j$ of orbits:
the finite set onto which we project has cardinality $j$.
When there is just one orbit, then it is called \textit{transitive} (which is thus a particular case of BRW projected on a
branching process).


To  a generic discrete-time BRW we associate a graph $(X,E_\mu)$ where $(x,y) \in E_\mu$  
if and only if $m_{xy}>0$.
We say that there is a path from $x$ to $y$ of length $n$, and we write $x \stackrel{n}{\to} y$, if it is
possible to find a finite sequence $\{x_i\}_{i=0}^n$ (where $n \in \N$)
such that $x_0=x$, $x_n=y$ and $(x_i,x_{i+1}) \in E_\mu$
for all $i=0, \ldots, n-1$. Clearly $x \stackrel{0}{\to}x$ for all $x \in X$; if there exists $n \in \N$
such that $x \stackrel{n}{\to} y$ then we write $x \to y$.
If the graph $(X,E_\mu)$ is \textit{connected} 
then we say that the BRW 
is \textit{irreducible}. 

%

We consider initial configurations with only one particle placed at a fixed site $x$ and we
denote by $\pr^{\delta_{x}}$ the law of the corresponding process. 
We now distinguish between the possible long-term behaviours of a BRW.
\begin{defn}\label{def:survival} $\ $
\begin{enumerate}
 \item 
The process \textsl{survives locally 
} in $A \subseteq X$ starting from $x \in X$ if
$
{\mathbf{q}}(x,A)
:=1-\pr^{\delta_x}(\limsup_{n \to \infty} \sum_{y \in A} \eta_n(y)>0)<1.
$
\item
The process \textsl{survives globally 
} starting from $x$ if
$
\bar {\mathbf{q}}(x)
:={\mathbf{q}}(x,X) 
<1$.
\end{enumerate}
\end{defn}
\noindent
In the rest of the paper we use the notation ${\mathbf{q}}(x,y)$ instead of ${\mathbf{q}}(x, \{y\})$ for all $x,y \in X$.
When there is no survival with positive probability, we say that there is extinction
and the fact that extinction occurs 
almost surely
will be tacitly understood. 
It is worth noting that, in the irreducible case, for every $A \subseteq X$, the inequality ${\mathbf{q}}(x,A)<{1}$ holds for some $x \in X$ if and only if it
holds for every $x \in X$.
For details and results on survival and extinction see for instance Section~\ref{sec:extprobvect} or \cite{cf:BZ4, cf:Z1}.

%

\subsection{Continuous-time branching random walks}
\label{subsec:continuous}

In continuous time each particle has an exponentially distributed
random lifetime with parameter 1. The breeding mechanisms can
be regulated by putting on each couple $(x,y)$ and for each particle at $x$,
a clock with $Exp(\lambda k_{xy})$-distributed intervals (where $\lambda>0$),
each time the clock
rings the particle breeds in $y$.

%

If one is only interested in survival and extinction of the process, the
continuous-time BRW
has a discrete-time counterpart with the same long-term behavior:
here is the construction.
The initial particles represent the generation $0$ of the discrete-time BRW;
the generation $n+1$ (for all $n \ge 0$) is obtained by
considering the children of all the particles of generation $n$
(along with their positions).


If $X$ has a graph structure and $(k_{xy})_{x,y \in X}$ is the adjacency matrix 
then we call the process an \textit{edge-breeding BRW}; in this case the
graph $(X, E_{\mu})$ associated to the discrete-time counterpart is the preexisting
graph on $X$. In particular an edge-breeding BRW is quasi-transitive if and only if the underlying
graph is.

Given $x \in X$, two critical parameters are associated to the
continuous-time BRW: the global 
survival
critical parameter $\lambda_w(x)$ and the local 
survival one $\lambda_s(x)$.
They are defined as
\begin{equation}\label{eq:criticalparameters}
\begin{split}
\lambda_w(x)
:=\inf\{\lambda>0:\,\pr^{\delta_{x}}\left(\exists t:\eta_t=\mathbf{0}\right)<1\}, \quad 
\lambda_s(x)
:=\inf\{\lambda>0:\,\pr^{\delta_{x}}\left(\exists \bar t:\eta_t(x)=0,\,\forall t\ge\bar t\right)<1\}.
  \end{split}
\end{equation}
n particular when
$\lambda < \lambda_w(x)$ (resp.~$\lambda >\lambda_w(x)$) then $\bar{\mathbf{q}}(x)=1$ (resp.~$\bar{\mathbf{q}}(x)<1$);
while if $\lambda = \lambda_w(x)$ there can be both global extintion and global survival (see for instance \cite{cf:BZ2}). 
As for the local behavior $\lambda \le \lambda_s(x)$ if and only if $\mathbf{q}(x,x)=1$.
If the process is irreducible then the critical parameters do not depend on $x$. 
See \cite{cf:BCZ, cf:BZ,cf:BZ2, cf:BZ4} for a more detailed
discussion on the values of $\lambda_w(x)$ and $\lambda_s(x)$, including their characterizations.

\subsection{Infinite-dimensional generating function}\label{subsec:genfun}

To the family $\{\mu_x\}_{x \in X}$, we associate the following generating function $G:[0,1]^X \to [0,1]^X$,
\[ 
G({\mathbf{z}}|x):= \sum_{f \in S_X} \mu_x(f) \prod_{y \in X} {\mathbf{z}}(y)^{f(y)},
\] 
where $G({\mathbf{z}}|x)$ is the $x$ coordinate of $G({\mathbf{z}})$.
%
%
The family $\{\mu_x\}_{x \in X}$ is uniquely determined by $G$. 
$G$ is continuous with respect to the \textit{pointwise convergence topology} of $[0,1]^X$  and nondecreasing
with respect to the usual partial order of $[0,1]^X$ (see \cite[Sections 2 and 3]{cf:BZ2} for further details).
Extinction probabilities are fixed points
of $G$ and the smallest fixed point is $\bar {\mathbf{q}}$ (see Section~\ref{sec:extprobvect} for details):
more generally, given a solution of $G(\mathbf{z}) \le \mathbf{z}$ then $\mathbf z \ge \bar {\mathbf{q}}$.

Consider now the closed sets $F_G:=\{\mathbf{z} \in [0,1]^X \colon G(\mathbf{z})=\mathbf{z}\}$, 
$U_G:=\{\mathbf{z} \in [0,1]^X \colon G(\mathbf{z}) \le \mathbf{z}\}$ and $L_G:=\{\mathbf{z} \in [0,1]^X \colon G(\mathbf{z})\ge \mathbf{z}\}$;
clearly $F_G = U_G \cap L_G$. Moreover, by the monotonicity property, $G(U_G) \subseteq U_G$ and $G(L_G) \subseteq L_G)$.
The iteration of $G$ produces sequences converging to fixed points.

\begin{pro}\label{pro:closed}
 Consider a sequence $\{\mathbf{z}_n\}_{n \in \N}$ in $[0,1]^X$ such that $\mathbf{z}_{n+1}=G(\mathbf{z}_n)$ for all $n \in \N$ and suppose that  
 $\mathbf{z}_n \to \mathbf{z}$ as $n \to +\infty$ for some $\mathbf{z} \in [0,1]^X$. Then 
 $\mathbf{z} \in F_G$.  
 Moreover, fix $\mathbf{w} \in [0,1]^X$. 
\begin{enumerate}
 \item If $\mathbf{w} \in U_G$ 
 then $\mathbf{w} \ge \mathbf{z}_0$ implies $\mathbf{w} \ge \mathbf{z}$ (the converse holds for 
 $\mathbf{z}_0 \in L_G$).
 \item If $\mathbf{w} \in L_G$ 
 then $\mathbf{w} \le \mathbf{z}_0$ implies $\mathbf{w} \le \mathbf{z}$ (the converse holds for 
 $\mathbf{z}_0 \in U_G$).
\end{enumerate} 
 \end{pro}
 The proof is straightforward (see for instance \cite{cf:BZ2}). 
 The convergence of the sequence $\{\mathbf{z}_n\}_{n \in \N}$ defined in the previous proposition holds 
 if $\mathbf{z}_0 \in L_G$ (resp.~$\mathbf{z}_0 \in U_G$): in that case $\mathbf{z}_n \uparrow \mathbf{z}$ (resp.~$\mathbf{z}_n \downarrow \mathbf{z}$) 
 for some $\mathbf{z} \in F_G$.
 
 The following properties of $U_G$ and $L_G$ allow to identify potentially new fixed points:
if we have a collection $\{\mathbf{w}_i\}_{i \in I}$ where $\mathbf{w}_i \in U_G$ (resp.~$\mathbf{w}_i \in L_G$) for all
$i \in I$ and $\mathbf{z}_0(x):=\inf_{i \in I} \mathbf{w}_i(x)$  then 
$\mathbf{z}_0 \in U_G$ (resp.~if $\mathbf{z}_0(x):=\sup_{i \in I} \mathbf{w}_i(x)$ then $\mathbf{z}_0 \in L_G$); for instance it is enough to
consider a collection $\{\mathbf{w}_i\}_{i \in I}$  of fixed points. In both cases $\mathbf{z}=\lim_{n \to +\infty} \mathbf{z}_n$ is well defined; moreover
if $\mathbf{z}_0<\mathbf{w}_i$ (resp.~$\mathbf{z}_0>\mathbf{w}_i$) for all 
$i \in I$ then $\mathbf{z}$ is a fixed point different from $\mathbf{w}_i$ for any $i \in I$. 

\subsection{Fixed points and extinction probabilities}\label{sec:extprobvect}

Define ${\mathbf{q}}_n(x,A)$
as the probability of extinction in $A$ before time $n$ starting with one particle at $x$, namely
${\mathbf{q}}_n(x,A)=\Prob^{\delta_x}(\eta_k(y)=0,\, \forall k\ge n,\,\forall y\in A)$. The sequence
$\{{\mathbf{q}}_n(x,A)\}_{n \in \N}$ is nondecreasing and satisfies
\begin{equation}\label{eq:extprobab}
 \begin{cases}
 \mathbf{q}_{n}(\cdot,A)=G(\mathbf{q}_{n-1}(\cdot, A)),& \quad \forall n \ge 1\\
 \mathbf{q}_0(x,A)=0, &\quad \forall x \in A, \\
 \mathbf{q}_0(x,A)=G(\mathbf{q}_0(\cdot,A)|x) &\quad  \forall x \not \in A,\\
\end{cases}
\end{equation}
Moreover,  ${\mathbf{q}}_n(x, A)$
converges to ${\mathbf{q}}(x,A)$,
which is the probability of local extinction in $A$
starting with one particle at $x$ (see
Definition~\ref{def:survival}); more precisely ${\mathbf{q}}_n(\cdot, A)\in L_G$ for all $n \in \N$. 
Since $G$ is continuous, by Proposition~\ref{pro:closed} we have that ${\mathbf{q}}(\cdot,A)=G({\mathbf{q}}(\cdot,
A))$, hence these extinction probability vectors are
fixed points of $G$.
For details on the last equality in equation~\eqref{eq:extprobab} see Remark~\ref{rem:q0}. 
We denote the set of extinction probability vectors by $E_G:=\{\mathbf{q}(\cdot, A)\colon A \subseteq X\} \supseteq \{\bar{\mathbf{q}}, \mathbf{1}\}$,
since $\mathbf{q}(\cdot,\emptyset)=\mathbf{1}$ where $\mathbf{1}(x)=1$ for all $x \in X$.

Clearly $U_G \supseteq F_G \supseteq E_G$ and 
it is well known that $\bar {\mathbf{q}}$ is the smallest element of each one of these three sets  
(since $\mathbf{q}_0(\cdot,X)=\mathbf{0}$, it is enough to apply Proposition~\ref{pro:closed}) and $\mathbf{1}$ is the largest one.
Hence $\bar {\mathbf{q}}= {\mathbf{1}}$ (global extinction) if and only if at least one of these sets is a singleton, that is, 
if and only if they are all singletons.

Note that $A \subseteq B \subseteq X$ implies ${\mathbf{q}}(\cdot,A)
\ge {\mathbf{q}}(\cdot,B)
\ge \bar {\mathbf{q}}$. 
Since for all finite $A\subseteq X$ we have ${\mathbf{q}}(x,A) \ge 1-\sum_{y \in A} (1-{\mathbf{q}}(x,y))$
then, for any given finite $A \subseteq X$, ${\mathbf{q}}(x,A)=1$ if and only if ${\mathbf{q}}(x,y)=1$ for all $y \in A$.
If the BRW is irreducible then, for all $A \subseteq X$,  ${\mathbf{q}}(\cdot,A)<\mathbf{1}$ if and only if ${\mathbf{q}}(x,A)<1$ for all $x \in X$;
moreover for all finite (nonempty) subsets $A,B \subseteq X$ we have
$\mathbf{q}(\cdot, A)=\mathbf{q}(\cdot, B)$. 
%

\begin{rem}\label{rem:q0}
 We observe that if $d(x,A):=\min\{n \in \N, y \in A \colon x \stackrel{n}{\to}y\}$ then
 $\mathbf{q}_n(x,A)=\mathbf{q}_0(x,A)$ for all $x$ such that $d(x,A) \ge n$. Hence,
 $\mathbf{q}_1(x,A)=\mathbf{q}_0(x,A)$ for all $x \not \in A$ and according to equation~\eqref{eq:extprobab}
 we have $\mathbf{q}_0(x,A)=G(\mathbf{q}_0|x)$ for all $x \not \in A$.
\end{rem}

\subsection{Extinction probabilities in different sets.}
We give here a couple of results which allow, in some cases, to know whether  
${\mathbf{q}}(x, A)$ is different from ${\mathbf{q}}(x, B)$.
The first theorem, is a generalization of \cite[Theorem 3.3]{cf:BZ14-SLS} and \cite[Theorems 4.1 and 4.2]{cf:BZ2017}.
We include the proof for the sake of completeness.
In the case of global survival, it gives equivalent conditions for strong local survival
in terms of extinction probabilities.
\begin{teo}\label{th:strongconditioned}
\begin{enumerate}[a)]
\item For every 
subset $A \subseteq X$ and every fixed point $\mathbf{z}$ of $G$, the following assertions are equivalent.
\begin{enumerate}[(1)]
\item  ${\mathbf{q}}(x, A) \le \mathbf{z}(x)$, for all $x \in X$;
\item ${\mathbf{q}}_0(x,A) \le \mathbf{z}(x)$, for all $x \in X$.
\end{enumerate}
In particular if $\mathbf{z}= {\mathbf{q}(\cdot,B)}$, where $A \subseteq B$, then the previous conditions are equivalent to 
(3) ${\mathbf{q}}(x,A) = {\mathbf{q}}(x, B)$ for all $x \in X$.
\item Consider two BRWs  $(X,\mu)$ and $(X,\nu)$. Suppose that  $A \subseteq X$ is a nonempty set
 such that $\mu_x=\nu_x$ for all $x \not \in A$. Then ${\mathbf{q}^\mu_0}(\cdot, A) = {\mathbf{q}^\nu_0}(\cdot, A)$
 Moreover, if $A \subseteq B$ then
 \[ 
  {\mathbf{q}^\mu}(x, A) = \mathbf{q}^\mu(x,B), \   \forall x \in X, \Longleftrightarrow
  {\mathbf{q}^\nu}(x, A) = \mathbf{q}^\nu(x,B), \  \forall x \in X.
 \]
\end{enumerate}
\end{teo}

\begin{proof}
\begin{enumerate}[a)]
 \item  If ${\mathbf{q}}_0(\cdot,A) \le \mathbf{z}(\cdot)$ then, by equation~\eqref{eq:extprobab}, ${\mathbf{q}}_n(\cdot,A) \le \mathbf{z}(\cdot)$ for all $n \in \N$,
 whence $\mathbf{q}(\cdot, A)=\lim_{n \to +\infty} \mathbf{q}_n(\cdot,A) \le \mathbf{z}(\cdot)$.
 Conversely, if ${\mathbf{q}}(\cdot,A) \le \mathbf{z}(\cdot)$ then, by the monotonicity of $\{{\mathbf{q}}_n(\cdot,A)\}_{n \in \N}$  we have
 ${\mathbf{q}_0}(\cdot,A) \le \mathbf{z}(\cdot)$.
 \item The equality ${\mathbf{q}^\mu_0}(x, A) = {\mathbf{q}^\nu_0}(x, A)$ is trivial when $x \in A$ and, when $x \not \in A$,
 it follows from the fact that the behavior of the two BRWs is the same until they first hit the set $A$. 
 From the previous part of the theorem, by taking $\mathbf{z}=\mathbf{q}^\mu(\cdot,B)$ and $\mathbf{z}=\mathbf{q}^\nu(\cdot,B)$ we have that  
 \[
  \begin{split}
 {\mathbf{q}^\mu}(x, A) = \mathbf{q}^\mu(x,B), \   \forall x \in X, & \Longleftrightarrow {\mathbf{q}^\mu_0}(x, A) \le \mathbf{q}^\mu(x,B),  \forall x \in X;  \\
 {\mathbf{q}^\nu}(x, A) = \mathbf{q}^\nu(x,B), \   \forall x \in X, &\Longleftrightarrow {\mathbf{q}^\mu_0}(x, A) \le \mathbf{q}^\nu(x,B),  \forall x \in X.\\
  \end{split}
 \]
  If, for some $x \in X$, $  {\mathbf{q}^\mu}(x, A) \not = \mathbf{q}^\mu(x,B)$ then, since $A \subseteq B$, $  {\mathbf{q}^\mu}(x, A) > \mathbf{q}^\mu(x,B)$; thus,
  from the previous part of the theorem, with positive probability $(X, \mu)$ survives in $B$ without ever visiting $A$ (starting from a suitable 
  $y \in B\setminus A$). Thus, the same holds for $(X,\nu)$
  (because their behavior is the same until they first hit $A$),
  thus, ${\mathbf{q}^\nu}(x, A) > \mathbf{q}^\nu(x,B)$ for some $x \in X$.  By switching, now, the roles of $(X, \mu)$ and $(X, \nu)$, the equivalence follows.
\end{enumerate}
\end{proof}


%

From the previous theorem, we have the following dichotomy: for every sets $A, B  \subseteq X$, either 
${\mathbf{q}}(\cdot,A) \le {\mathbf{q}}(\cdot, B)$ or 
there is $x \in B\setminus A$ such that there is a positive probability of local survival in $B$ starting from $x$ without ever visiting $A$.
In particular ${\mathbf{q}}(x,A)>{\mathbf{q}}(x, B)$ implies that there is
a positive probability of local survival in $B$ and
local extinction in $A$ starting from $x$ (if $A \subseteq B$ then also the converse is true).
Note that, ${\mathbf{q}}_0(x,A)>{\mathbf{q}}(x, B)$ implies ${\mathbf{q}}(x,A)>{\mathbf{q}}(x,B)$ but 
the converse is not true.
The second tool that we need is the following lemma.

\begin{lem}\label{lem:smallersurvival}
 Consider a BRW $(X,\mu)$ and three subsets $A_1, A_2 \subseteq A \subseteq X$ such that $A_1 \cap A_2 =\emptyset$. 
If there exists $z \in X$ such that $\pr^z(\sum_{x \in A_1} \eta_n(x) >0 \  i.o., \lim_{n \to +\infty} \sum_{x \in A_2} \eta_n(x)=0)>0$ then  
$\mathbf{q}(z,A_2) > \mathbf{q}(z, A)$, whence $\mathbf{q}(\cdot,A_2) > \mathbf{q}(\cdot, A)$.

 In particular, if  $x$ is such that $x \to z$ then $\mathbf{q}(x,A_2) > \mathbf{q}(x,A)$ provided that $\bar{\mathbf{q}}(y)>0$
 whenever $x \to y$ (for instance, if $\mu_y(\mathbf{0})>0$ for all $y$).
\end{lem}

\begin{proof}
 From the inclusion $A_i \subseteq A$ we have $\mathbf{q} (\cdot, A_i) \ge \mathbf{q}(\cdot, A)$ (for $i=1,2$). From the hypotheses we have
 \[
 \pr^z(\sum_{x \in A} \eta_i(x) >0 \  i.o., \lim_{i \to +\infty} \sum_{x \in A_2} \eta_i(x)=0) \ge 
 \pr^z(\sum_{x \in A_1} \eta_n(x) >0 \  i.o., \lim_{n \to +\infty} \sum_{x \in A_2} \eta_n(x)=0) >0 
 \]
 whence $\mathbf{q}(z,A_2) > \mathbf{q}(z,A)$ and this implies $\mathbf{q}(\cdot, A_2) > \mathbf{q}(\cdot, A)$.

If $x \to z$, then there is a positive probability $p_0$ that the process can reach $z$ and that the progenies
of all the particles, except at most one at $z$, die out; in this case the long term behavior is given by the evolution
of the progeny of one particle at $z$.
Thus $\pr^x(\sum_{y \in A_1} \eta_i(y) >0 \  i.o., \lim_{i \to +\infty} \sum_{y \in A_2} \eta_i(y)=0) \ge 
p_0 \pr^z(\sum_{y \in A_1} \eta_n(A_1) >0 \  i.o., \lim_{n \to +\infty} \sum_{y \in A_2} \eta_n(y)=0) >0$.
As before, this implies $\mathbf{q}(x,A_2) > \mathbf{q}(x,A)$.
\end{proof}


\section{BRWs with uncountably many extinction probabilities}\label{sec:main}

%

\subsection{A BRW with an uncountable set of extinction probability vectors: the tree} \label{exmp:tree}
 Consider the regular tree $\mathbb{T}_m$ (where $m \ge 3$) 
 and the discrete-time counterpart of a continuous-time BRW where $K$ is the adjacency matrix on $\mathbb{T}_m$; 
 for this BRW it is well known that $\lambda_w=1/m<1/2\sqrt{m-1}= \lambda_s$.
 Denote a vertex by $o$ and call it the root. Given $x \neq o$ we denote by $T_x$ the subtree branching from $x$, that is, the set of vertices 
 which are disconnected from $o$ by removing $x$ from $\mathbb{T}_m$; moreover, let $T_o:=\mathbb{T}_m$.
  Given any automorphism $\Psi$  of $\mathbb{T}_m$ (that is, a bijective map preserving the edges), 
 one can easily prove that $\mathbf{q}(\cdot,A)=\mathbf{q}(\Psi(\cdot),\Psi(A))$;
 in particular if $\Psi(o)=o$ and $\Psi(x)=y$ then $\Psi(T_x)=\Psi(T_y)$, thus
$\mathbf{q}(\cdot, T_x)=\mathbf{q}(\Psi(\cdot), T_y)$.

If $\lambda \le \lambda_w$ then there is only one fixed point, namely $\mathbf{z}=\mathbf{1}$;
if $\lambda > \lambda_s$ then there are just two extinction probability vectors, $\mathbf{q}(\cdot, \mathbb{T}_3)$ and $\mathbf{1}$, indeed in this case 
$\mathbf{q}(\cdot, A)=\mathbf{q}(\cdot, \mathbb{T}_m)$ for all $A \not = \emptyset$ (see \cite[Corollary 3.2]{cf:BZ14-SLS} and \cite[Example 4.5]{cf:BZ2017}).
 The last case $\lambda \in (\lambda_w, \lambda_s]$ is the most interesting one:  
 $\mathbf{q}(\cdot, \mathbb{T}_m)<\mathbf{1}$ while $\mathbf{q}(\cdot, A)=\mathbf{1}$ for every finite $A \subset \mathbb{T}_3$.
 We prove now that $x \mapsto \mathbf{q}(\cdot, T_x)$ is injective, thus $F_G$ is at least countable.

 Henceforth, for simplicity we consider just the case $m=3$ although an analogous construction can be done 
on $\mathbb{T}_m$ for all $m \ge 3$; indeed, a more general example is sketched in Section~\ref{subsec:comb} (see Theorem~\ref{thm:comb}). 
Let $d$ be the natural distance on the graph $\mathbb{T}_3$ and consider a sequence $\{y_n\}_{n \in \Z}$ of distinct vertices
such that $y_n$ is a neighbor of $y_{n+1}$ (for all $n \in \Z$) and $d(o, y_n)=|n|$ (clearly $y_0=o$).
We denote by $x_i$ the third neighbor of $y_{i-1}$ (outside $y_{i-2}$ and $y_i$).
A graphical representation is depicted in Figure~1. 
We note that $d(o,x_n)=n$ for all $n \in \N$ and $T_{x_n} \cap T_{x_m}=\emptyset$ whenever $n \neq m$.

\begin{center}
\begin{figure}\label{fig:t3tree}
 \includegraphics[width=1 \textwidth]{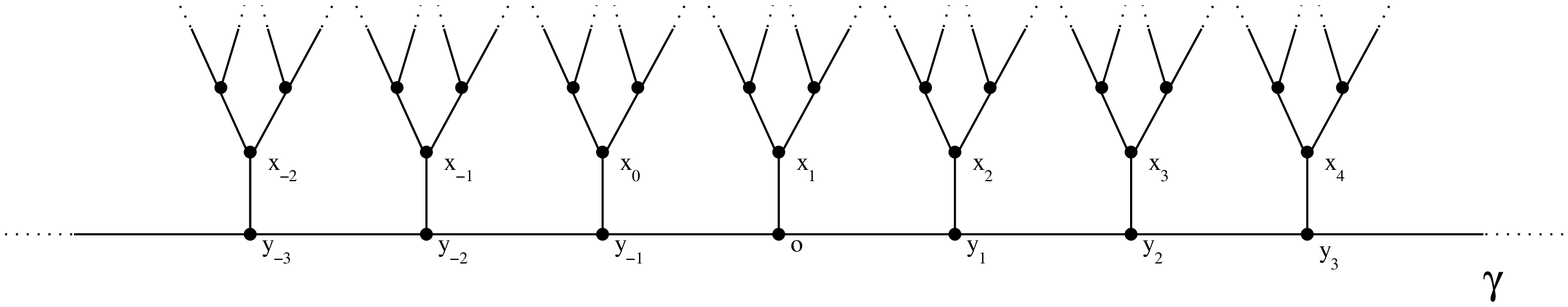}
 \caption{The tree $\mathbb{T}_3$}
\end{figure}
 \end{center}


\begin{lem}\label{lem:countable}
 Let $\lambda \in (\lambda_w, \lambda_s]$. For every $n \ge 1$ and $x \in \mathbb{T}_3$.
 \[
  \mathbf{q}(x, \mathbb{T}_3)=\mathbf{q}(x, T_{y_0}) < \mathbf{q}(x, T_{y_n}) < \mathbf{q}(x, T_{y_{n+1}})< 1.
 \]
\end{lem}

\begin{proof}
Denote by $A$ the subtree $T_{y_n}$, by $A_2$ the subtree $T_{y_{n+1}}$ and by $A_1$ the set $A \setminus (A_2 \cup \{y_n\})$.
Since $\mathbf{q}(\cdot, \mathbb{T}_3)<1$ and $\mathbf{q}(\cdot, \{y_n\})=1$ then, 
by Theorem~\ref{th:strongconditioned} (if we take $B:=\{y_n\}$ then (3) fails), there exists
$w \in \mathbb{T}_3$ such that the process starting from $w$ survives with positive probability without ever visiting $y_n$;
by rotational symmetry centered in $y_n$, $w$ can be chosen in $A_1$. 
If a process starts in $A_1$ and never visits $y_n$ it is then confined to $A_1$ whence, by
Lemma~\ref{lem:smallersurvival}, $\mathbf{q}(\cdot, T_{y_{n+1}}) > \mathbf{q}(\cdot, T_{y_{n}})$ and the strict inequality holds for every coordinate.
\end{proof}

Observe that, given $x,z \in  \mathbb{T}_3$ such that 
$z \not \in T_x$, then $\mathbf{q}(z, T_x)$ depends only on $d(z,x)$. Indeed,  there exists an automorphism $\Psi$ such that 
$\Psi(z)=o$ and $\Psi(x)=y_{d(z,x)}$. Thus, $\Psi(T_x)=T_{y_{d(z,x)}}$, hence $\mathbf{q}(z,T_x)=\mathbf{q}(\Psi(z),\Psi(T_x)) = \mathbf{q}(o, T_{y_{d(z,x)}})$.
In particular, if $x,w \in  \mathbb{T}_3$ are such that $d(o,x) < d(o,w)$ 
then $\mathbf{q}(o, T_x)=\mathbf{q}(o, T_{y_{d(o,x)}}) < \mathbf{q}(o, T_{y_{d(o,w)}})=\mathbf{q}(o, T_w)$
(the case $d(o,x) > d(o,w)$ is analogous).
If $d(o,x)=d(o,w)$ then $d(x,w)>0$ and 
it is easy to show that 
$\mathbf{q}(x, T_x) < \mathbf{q}(o, T_{y_{d(x,w)}})=\mathbf{q}(x, T_w)$.

Now we prove the local extinction on a bi-infinite line.

%

\begin{lem}\label{lem:extinctionline}
Let $\lambda \in (\lambda_w, \lambda_s]$. If $\gamma$ is a bi-infinite line in $\mathbb{T}_3$ then
$\mathbf{q}(x,\gamma)=1$.
\end{lem}

\begin{proof}
 It is enough to prove that $\mathbf{q}(x,\gamma)=1$ when 
 $\gamma:=\{y_n\}_{n \in \N}$.
 Since there is a.s.~local extinction then $\mathbf{q}(\cdot,\gamma)\ge \lim_{n \to \infty}\mathbf{q}(\cdot, T_{y_n})$.  Now,
by using a suitable automorphism, 
then $\mathbf{q}(o,T_{y_n})=\mathbf{q}(y_{n-1}, T_{x_1})$ for all $n >0$. 
Since $\mathbf{q}(\cdot, T_{x_1}) \neq \mathbf{q}(\cdot, X)$ then a result of Moyal 
(see \cite{cf:Moyal}) and the transitivity of $\mathbb{T}_3$
imply that $\lim_{n \to \infty}\mathbf{q}(y_{n-1}, T_{x_1})=1$ whence $\mathbf{q}(o,\gamma)=1$; the irreducibility yields
$\mathbf{q}(x,\gamma)=1$ for all $x \in \mathbb{T}_3$.
\end{proof}

There are two interesting consequences of the previous lemma.
\begin{enumerate}
 \item Any surviving population leaves a.s.~every bi-infinite line $\gamma$.
 \item Since $\mathbf{q}(\cdot, \gamma)=\mathbf{1}$ then the population visits a.s.~a finite number of vertices $\{y_n\}_{n \in \N}$, hence survival occurs
in a finite number of subtrees $\{T_{x_n}\}_{n \in \Z}$ 
(this argument can be repeated inside each subtree and so on). Thus, for all $I \subseteq \N \setminus \{0\}$ we have
$\mathbf{q}(\cdot, \bigcup_{i \in I\colon i \le n} T_{x_i}) \downarrow \mathbf{q}(\cdot, \bigcup_{i \in I} T_{x_i})$
as $n \to +\infty$.
\end{enumerate}

%

By Lemma~\ref{lem:countable} we have at least a countable collection of distinct extinction probability vectors. The following 
theorem proves the existence of an uncountable collection.

\begin{teo}\label{thm:uncountableEPVs}
 Let $\lambda \in (\lambda_w, \lambda_s]$. If $I_1, I_2 \subseteq \N \setminus\{0\}$ such that
 $\sum_{n \in I_1} 2^{-i} \neq \sum_{n \in I_2} 2^{-i}$ then 
 $\mathbf{q}(\cdot, \bigcup_{i \in I_1} T_{x_i}) \neq \mathbf{q}(\cdot, \bigcup_{i \in I_2} T_{x_i})$.
\end{teo}

\begin{proof}
We observe that $\mathbf{q}(o, T_{x_n})=\mathbf{q}(o, T_{y_n})$ for all $n \ge 1$.
We start by proving that for all finite $I\subseteq \N \setminus\{0\}$, if $i_0=\max I$, 
 \begin{equation}\label{eq:finiteinfinite}
    \mathbf{q}(o, \bigcup_{i \in I} T_{x_i})=
  \mathbf{q}(o, \bigcup_{i \in \bar I} T_{x_i})
 \end{equation}
where $\bar I:=\{i  \in I\colon i < i_o\} \cup \{i > i_0 \}$.
Indeed, by a simple automorphism argument (choose an automorphism $\Psi$ such that
$\Psi(o)=o$ and $\Psi(x_{i_0})=\Psi(y_{i_0})$), 
 $\mathbf{q}(o, \bigcup_{i \in I} T_{x_i})=
 \mathbf{q}(o, \bigcup_{i \in I\colon i < i_o} T_{x_i} \cup T_{y_{i_0}})$. Since
 there is a.s.~extinction in every infinite line, then survival in $T_{y_{i_0}}$ is
 equivalent to survival in $\bigcup_{i > i_0} T_{x_i}$ and this yields equation~\eqref{eq:finiteinfinite}.
 
 From Lemma~\ref{lem:smallersurvival} as in Lemma~\ref{lem:countable},
we have that for all $I, J \subseteq \N \setminus \{0\}$ 
\begin{equation}\label{eq:inclusion}
 I \subsetneqq J \Longrightarrow \mathbf{q}(o,\bigcup_{i \in I} T_{x_i}) > \mathbf{q}(o,\bigcup_{i \in J} T_{x_i}).
\end{equation}
 
 Since $I_1 \neq I_2$ we can define $i_0 :=\min I_1 \triangle I_2$; suppose, without loss of generality, that
 $i_0 \in I_1 \setminus I_2$.
 Define $I_3:=\{i \in I_1 \colon i \le i_0\}$ and $I_4:=\{i\in I_3\colon i<i_0\} \cup \{i > i_0\}$.
 Note that by equation~\eqref{eq:finiteinfinite},
 \begin{equation}\label{eq:I3I4}
  \mathbf{q}(o, \bigcup_{i \in I_3} T_{x_i})=
   \mathbf{q}(o, \bigcup_{i \in I_4} T_{x_i}).
 \end{equation}
Moreover, $I_3\subseteq I_1$ and $I_2\subseteq I_4$.
 Since 
 $\sum_{n \in I_1} 2^{-i} \ge \sum_{n \in I_3} 2^{-i} = \sum_{n \in I_4} 2^{-i} \ge \sum_{n \in I_2} 2^{-i}$
 but
 $\sum_{n \in I_1} 2^{-i} > \sum_{n \in I_2} 2^{-i}$ (remember that $i_0 \in I_1 \setminus I_2$) then
 we have just two possible cases.
 \begin{itemize}
  \item $I_2 \subsetneqq I_4$, $I_3 \subseteq I_1$. In this case, by equation~\eqref{eq:inclusion},
  \begin{equation}\label{eq:ineq1}
    \mathbf{q}(o, \bigcup_{i \in I_1} T_{x_i}) \le
  \mathbf{q}(o, \bigcup_{i \in I_3} T_{x_i})=
   \mathbf{q}(o, \bigcup_{i \in I_4} T_{x_i}) <
  \mathbf{q}(o, \bigcup_{i \in I_2} T_{x_i}).
  \end{equation}
   \item $I_2 = I_4$, $I_3 \subsetneqq I_1$.In this case, again by equation~\eqref{eq:inclusion},
  \begin{equation}\label{eq:ineq2}
    \mathbf{q}(o, \bigcup_{i \in I_1} T_{x_i}) <
  \mathbf{q}(o, \bigcup_{i \in I_3} T_{x_i})=
   \mathbf{q}(o, \bigcup_{i \in I_4} T_{x_i}) =
  \mathbf{q}(o, \bigcup_{i \in I_2} T_{x_i}).
  \end{equation}
 \end{itemize}

 \end{proof}

Note that the previous theorem contradicts what had been written in \cite[p.244]{cf:BZ14-SLS},
namely it is not true that on quasi-transitive irreducible BRWs there are at most two extinction 
probability vectors. What remains true is that either
$q(\cdot,A)=\mathbf 1$ for all finite $A$ or $q(\cdot,A)=\bar{\mathbf{q}}(\cdot)$ for all nonempty subsets $A$ (in particular if there is local survival
at $x$, then there is strong local survival at each $y\in X$). This implies that when $\bar{\mathbf{q}}<\mathbf{1}=\mathbf q (\cdot, A)$ for all finite 
subsets $A$ then in general nothing can be said about $\mathbf q (\cdot, A)$ when $A$ is infinite: in \cite[Example 3.6]{cf:BZ2017} 
$\mathbf q (\cdot, A)=\mathbf q(\cdot, X)$ for every infinite $A$, in \cite[Examples 1 and 2]{cf:BraHautHessemberg2} there are BRWs with
a finite number of extinction probability vectors corresponding to different choices of the infinite set $A$,
while in the above BRW on the tree there are uncountably many different
extinction probability vectors.

An uncountable set of of extinction probability vectors can also be found in BRWs where there is local survival as the following example shows.
\begin{exmp}\label{exmp:modifiedBRW}
Consider the BRW on the tree obtained by adding a loop at $o$. If the loop has a sufficiently large reproduction rate, the BRW has 
local survival at every vertex  (see \cite[Example 4.5]{cf:BZ2017}). It also has 
an uncountable number of extinction probability vectors. Indeed, in order to apply Lemma~\ref{lem:smallersurvival} 
to obtain Lemma~\ref{lem:countable}, we
just need to prove that there is a positive probability of surviving in $T_{y_{n+1}}$ without ever visiting $y_n$: 
this is equivalent to surviving in $B:=T_{y_{n+1}} \cup \{o\}$ without ever visiting $A:=\{y_{n}, o\}$ and 
this follows from Theorem~\ref{th:strongconditioned}. 
Moreover, in case Lemma~\ref{lem:extinctionline}
does not hold, the equality in equation~\eqref{eq:finiteinfinite} becomes an inequality which implies 
$ \mathbf{q}(o, \bigcup_{i \in I_3} T_{x_i}) \le
   \mathbf{q}(o, \bigcup_{i \in I_4} T_{x_i})$ instead of the equality in equation~\eqref{eq:I3I4}, but this does not change the conclusions in 
   equations~\eqref{eq:ineq1}~and~\eqref{eq:ineq2}.
\end{exmp}

Theorem~\ref{thm:uncountableEPVs} implies that 
the relation $\{(\sum_{i \in I}  2^{-i}, \mathbf{q}(o, \bigcup_{i \in I} T_{x_i}))\}_{I \subseteq \N \setminus \{0\}}$ 
is a well-defined,
strictly decreasing map. 
One may conjecture that $q(x,A)$ only depends on ``how large $A$ is at infinity''. To be more precise,
%
consider the simple random walk and the branching random walk (with rates 1 on each edge) on the regular tree $\mathbb{T}_m$.
Denote by $(\mathcal{M}, \nu_x)$ the measure space where $\mathcal{M}$ is the Martin boundary and $\nu_x$ is the harmonic measure related to the
random walk starting from $x \in \mathbb{T}_m$. For any $A \subseteq \mathbb{T}_m$ there is a well-defined (possibly empty) boundary $\partial A
\in \mathcal{M}$ (see \cite{cf:Woess} for details on the Martin boundary of a random walk and the associated measure).
Is the relation $\{(\nu_x(\partial A), \mathbf{q}(x, A))\}_{A \subseteq \mathbf{T}_m}$ a well-defined map? 

The answer is no, by the following argument.
It is enough to prove that $q(o,A)\neq q(o,B)$ for $A$ and $B$ such that $\nu_o(\partial A)=\nu_o(\partial B)$.
Let $A=T_{x_1}$ and let $s$ and $v$ be the two neighbors of $x_1$ which are in $A$.
Since there is local extinction, $q(o,A)=q(o,A_1\cup A_2)$ where $A_1=T_s$ and $A_2=T_v$.
Let $B=A_2\cup T_{x_2}$. We focus on survival in $A$ and in $B$: it suffices to prove that
 the probability of survival in $A$ is different from the one os survival in $B$.
If the process survives in $A_2$ then there is survival both in $A$ and in $B$. We prove that
the probability of the event $\mathcal A^*$ = surviving in $A_1$ but not in $A_2$ is different from the probability
of the event $\mathcal B^*$ = surviving in $T_{x_2}$ and not in $A_2$.
Let $C_n^z$ be the event that the original particle at $o$ has exactly $n$ descendants at $z$, whose reproduction trail
hits $z$ for the first time (roughly speaking, the reproduction trail is the path which traces the lineage, see \cite{cf:PemStac1}
for a formal definition).
By simmetry, $\Prob(C_n^{x_1})=\Prob(C_n^{y_1})$, moreover $\mathcal A^*\subseteq\bigcup_{n\ge1} C_n^{x_1}$
and $\mathcal B^*\subseteq\bigcup_{n\ge1} C_n^{y_1}$.
Again by simmetry, $\mathcal A^*\cap C_n^{x_1}$ is the event where none of the $n$ children has an infinite progenies in $A_2$,
while at least one of them has an infinite progenies in $A_1$ and its probability is equal to the probability
that, starting with $n$ particles at $o$, there is extinction in $T_{y_1}$ and survival in $A$.
Similarly, $\mathcal B^*\cap C_n^{y_1}$ is the event where none of the $n$ children has an infinite progenies in $A_2$,
while at least one of them has an infinite progenies in $T_{x_2}$ and its probability is equal to the probability
that, starting with $n$ particles at $o$, there is extinction in $T_{y_3}$ and survival in $A$.
By Lemma~\ref{lem:countable}, $q(o,T_{y_1})<q(o,T_{y_3})$ and the same inequality holds for the process starting 
with $n$ particles. Since the event ``extinction in $T_{y_1}$'' is a subset of ``extinction in $T_{y_3}$'',
it is enough to note that the event ``extinction in $T_{y_3}$ with survival both in $A$ and in $T_{y_1}$''
has a positive probability.


\subsection{A BRW with an uncountable set of extinction probability vectors: the comb} \label{subsec:comb}

In this section we sketch the proof of a generalization of Theorem~\ref{thm:uncountableEPVs}.
To this aim, consider the BRW in Figure~2 on the 2-dimensional comb $C_2$, that is, the graph on $\{(x,y) \in \Z^2\colon y \ge 0\}$ where 
$(x,y)$ and $(x_1,y_1)$ are neighbors if and only if either ``$x=x_1$ and $|y-y_1|=1$'' or ``$y=y_1=0$ and $|x-x_1|=1$''.
Let $\alpha \ge 1$ and consider the rates $k_{xy}$ as in Figure~2, that is $1$ on the horizontal neighbors, $1$ downward and $\alpha+1$ upward except when
 leaving the horizontal axis where the rate is $\alpha$. We denote by $V_i$ the vertical line from $y_i$: when $i \ge 1$ this is $T_{y_i} \setminus T_{y_{i+1}}$.
  
  The following definition of \textit{projection of a BRW} 
first appeared in \cite{cf:BZ} for multigraphs, in \cite{cf:BZ2} for continuous-time BRWs and
\cite{cf:Z1} for generic discrete-time BRWs (in these papers it was called \textit{local isomorphism}). We just need it in the 
case of a continuous-time process.
   \begin{center}
\begin{figure}\label{fig:comb3}
 \includegraphics[width=1 \textwidth]{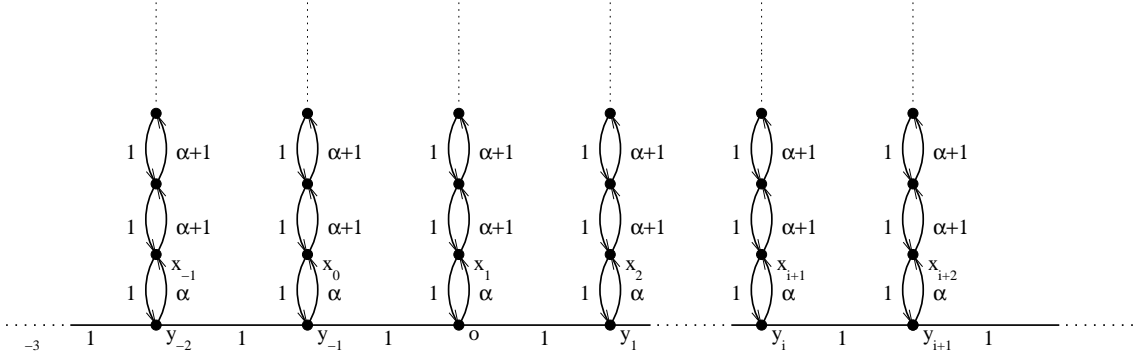}
\caption{The BRW on the comb}
\end{figure}
 \end{center}
 \begin{defn}\label{def:projection}
A projection of a BRW $(X,K)$ onto $(Y,\widetilde K)$ is a surjective map $g:X \to Y$, such that
$\sum_{z \in g^{-1}(y)} k_{xz}=\widetilde k_{g(x)y}$
for all $x \in X$ and $y \in Y$.
\end{defn}
If $\{\eta_t\}_{t \ge 0}$ is a realization of the BRW $(X, K)$, then $\{\sum_{z\in g^{-1}(\cdot)}\eta_t(z)\}_{t \ge 0}$ is a realization
of the BRW $(Y, \widetilde K)$. 
In particular it is easy to prove that $\mathbf{q}(x, g^{-1}(A))=\widetilde{\mathbf{q}}(g(x), A)$ for all $x \in X$ and $A \subseteq Y$.

To give an explicit example consider the BRW on the comb: this can be projected on a continuous-time branching process, that is
a BRW on a singleton with rate $\alpha+2$. This implies that $\lambda_w=1/(\alpha+2)$ while, by applying \cite[Proposition 4.3.3]{cf:BZ4},
$\lambda_s=1/(s \sqrt{\alpha+1})$. We note that the edge-breeding BRW on $\mathbb{T}_m$ can be projected on the BRW on the comb with $\alpha=m-2$
(one can easily understand it by comparing Figures~1 and~2).
%
%
One last example, which will be useful in the main result of the section is the BRW on the set $B$ in Figure~3: this BRW can be projected on the BRW on $V_i$
by a map $g: B \to V_i$ where $d(y_i, g(x))=d(y_i^\prime, x)$. 
  \begin{center}
\begin{figure}\label{fig:combB}
 \includegraphics[width=0.5 \textwidth]{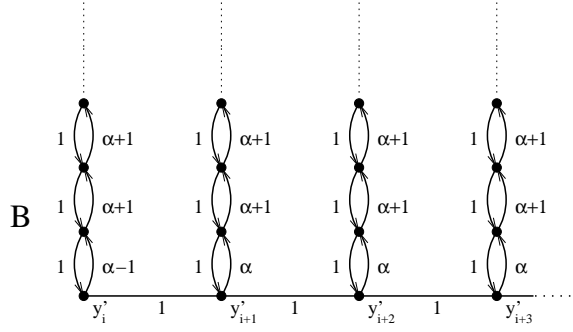}
\caption{The BRW on $B$}
\end{figure}
 \end{center}
 The following theorem can be considered a generalization of Theorem~\ref{thm:uncountableEPVs} in the sense that every BRW which can be projected
 on the comb, including the BRW on $\mathbb{T}_m$, satisfies the same property below (by using $g^{-1}(V_i)$ instead of $V_i$, where $g$ is the projection).
 
 \begin{teo}\label{thm:comb}
   Let $\lambda \in (\lambda_w, \lambda_s]$. If $I_1, I_2 \subseteq \N \setminus\{0\}$ such that
 $\sum_{n \in I_1} 2^{-i} \neq \sum_{n \in I_2} 2^{-i}$ then 
 $\mathbf{q}(\cdot, \bigcup_{i \in I_1} V_i) \neq \mathbf{q}(\cdot, \bigcup_{i \in I_2} V_i)$.
 \end{teo}
  \begin{proof}
We quickly sketch the main steps of the proof. We start by proving the analogous of Lemma~\ref{lem:countable} with $V_i$ instead of 
$T_{y_i}$. To this aim consider the BRW obtained from the BRW on $C_2$ by replacing $V_i$ with $B$ (we identify $y_i$ with $y_i^\prime$): we call
this BRW $C_2^\prime$. Clearly the BRW on $C_2^\prime$ can be projected on the BRW on $C_2$ by simply extending the function $g$ (defined above on $B$)
with the identity map on $C_2^\prime \setminus B \equiv C_2 \setminus V_i$. Clearly 
$\mathbf{q}^\prime (o, B)=\mathbf{q}^\prime (o, g^{-1}(V_i))=\mathbf{q}(o,V_i)$. But in $C_2^\prime$ there are both $V_{i+1}$ and a copy 
$V_{i+1}^\prime=T_{y_{i+2}^\prime}\setminus
T_{y_{i+1}^\prime}\subseteq B$
(the vertical line from $y_{i+1}^\prime$).
By using Lemma~\ref{lem:smallersurvival} we have
\[
 \mathbf{q}(o,V_{i+1})=\widetilde{\mathbf{q}}(o,V_{i+1})=\widetilde{\mathbf{q}}(o, V^\prime_{i+1})> \mathbf{q}^\prime (o, B) =\mathbf{q}(o,V_i)
\]
and from this $\mathbf{q}(x,V_{i+1})> \mathbf{q}(x,V_i)$ for all $x \in C_2$.

The last step, as in Example~\ref{exmp:modifiedBRW}, is to prove 
\[
     \mathbf{q}(o, \bigcup_{j \in I} V_j) \le
  \mathbf{q}(o, \bigcup_{j \in \bar I} V_j)
\]
where 
 $I\subseteq \N \setminus\{0\}$ is finite and  
$\bar I:=\{j  \in I\colon  < i\} \cup \{j > i \}$ ($i=\max I$). 
By using the projection on the BRW on $C_2^\prime$ we have that 
\[
 \mathbf{q}(o, \bigcup_{j \in \bar I} V_j)= \widetilde{ \mathbf{q}}(o, \bigcup_{j \in I\colon j <i} V_j \cup \bigcup_{j >i} V^\prime_j)
\ge \widetilde{\mathbf{q}}(o, \bigcup_{j \in I\colon j <i} V_j \cup B)=\mathbf{q}(o, \bigcup_{j \in I} V_i).
\]
Now we can use it, as in Example~\ref{exmp:modifiedBRW},
instead of equation~\eqref{eq:finiteinfinite} to prove an analogous inequality instead of the equality~\eqref{eq:I3I4}.
The claim then follows easily. 
  
%

 \end{proof}

\section{State of the art and open questions}\label{sec:open}

Let us summarize the (main) known relations between $\{\bar{\mathbf{q}}, \mathbf{1}\}$, $E_G$ and $F_G$ in the irreducible case. 
In between we list some interesting questions that, up to our knowledge, are still open.
Henceforth, we denote the cardinality
of a set by $| \cdot |$.

\noindent\textbf{$\bullet$} As we already noted, $F_G \supseteq E_G \supseteq \{\bar{\mathbf{q}}, \mathbf{1}\}$; moreover
$\bar{\mathbf{q}}=\mathbf{1} \Longleftrightarrow |F_G|=1 \Longleftrightarrow |E_G|=1$. This is equivalent to global extinction.

\noindent\textbf{$\bullet$} $X$ finite $\Longrightarrow F_G=E_G=\{\bar{\mathbf{q}}, \mathbf{1}\}$ (see, for instance, \cite{cf:Spataru89}
or \cite[Corollary 3.4]{cf:BZ14-SLS}). 

\noindent\textbf{$\bullet$} $X$ infinite, $(X,\mu)$ quasi-transitive and $\mathbf{q}(\cdot, x)<\mathbf{1}$ for some $x \in X$ 
$\Longrightarrow E_G=\{\bar{\mathbf{q}}, \mathbf{1}\}$ (here local survival implies strong local survival). 
Indeed, in this case, $\mathbf{q}(\cdot,A)=\bar{\mathbf{q}}$ for all $A \neq \emptyset$.
Whether the cardinality $|F_G \setminus E_G|$ can be positive (finite, countable or uncountable) is an open question.

\noindent\textbf{$\bullet$} $X$ infinite, $(X,\mu)$ quasi-transitive and $\mathbf{q}(\cdot, x)=\mathbf{1}$ for all $x \in X$: our
example in Section~\ref{exmp:tree} shows that $E_G$ can be uncountable. 
The cardinality $|F_G \setminus E_G|$ is unknown.
We conjecture that it can be uncountable, at least when $E_G$ is finite; indeed, we believe that \cite[Example 3.6]{cf:BZ2017} can be extended as explained
in \cite[Remark 3.7]{cf:BZ2017}.

\noindent\textbf{$\bullet$} $X$ infinite and $\mathbf{q}(\cdot, x)<\mathbf{1}$: Example~\ref{exmp:modifiedBRW} shows that 
$E_G$ can be uncountable. The cardinality $|F_G \setminus E_G|$ is unknown.

\noindent\textbf{$\bullet$} $X$ infinite, projected on a branching process: \cite[Example 3.6]{cf:BZ2017} shows that $F_G \setminus E_G$ 
can be uncountable ($E_G=\{\bar{\mathbf{q}}, \mathbf{1}\}$ in this case). 

\noindent\textbf{$\bullet$} In \cite{cf:BraHautHessemberg, cf:BraHautHessemberg2} there are examples of BRWs where either $|E_G|=|F_G|=2$ or
$E_G$ is finite (larger than $2$) and $F_G$ is uncountable.
\smallskip

Other interesting open questions on $|E_G|$ and $F_G$ are the following.

 \noindent\textbf{$\bullet$} Is it possible that $|E_G|<|F_G|<+\infty$?

  \noindent\textbf{$\bullet$} Is it possible that $E_G$ and $F_G\setminus E_G$ are both infinite?
  
\noindent In particular we conjecture that $E_G$ (resp.~$F_G$) is either finite or uncountable.


\section*{Acknowledgements}

The authors are grateful to Sophie Hautphenne for useful discussions.

\end{document}